\documentclass [11pt]{article}
\usepackage{fancyhdr}
\usepackage{graphicx}
\usepackage{framed}
\usepackage{amsmath}
\usepackage{amsfonts}
\usepackage{amsthm}
\usepackage{amssymb}
\newtheorem{theorem}{Theorem} [section]
\newtheorem{lemma}[theorem]{Lemma} 
 
\newtheorem{proposition}[theorem]{Proposition} 
\newtheorem{claim}[theorem]{Claim} 
\usepackage[font={small}]{caption}
\title{A lower bound for Torelli-K-quasiconformal homogeneity}
\date{}
\author{Mark Greenfield\thanks{Email: greenfield@caltech.edu. Department of Mathematics, California Institute of Technology, Pasadena, CA 91125.}
       }

\begin{document}
\lhead{M. Greenfield}
\chead{}
\rhead{}

\maketitle

\begin{abstract}  A closed hyperbolic Riemann surface $M$ is said to be $K$-quasiconformally homogeneous if there exists a transitive family $\mathcal{F}$ of $K$-quasiconformal homeomorphisms. Further, if all $[f] \subset \mathcal{F}$ act trivially on $H_1(M;\mathbb{Z})$, we say $M$ is Torelli-$K$-quasiconformally homogeneous. We prove the existence of a uniform lower bound on $K$ for Torelli-$K$-quasiconformally homogeneous Riemann surfaces. This is a special case of the open problem of the existence of a lower bound on $K$ for (in general non-Torelli) $K$-quasiconformally homogeneous Riemann surfaces. 
\end{abstract}

\section{Introduction}   
$K$-Quasiconformal homeomorphisms of a Riemann surface $M$ generalize the notion of conformal maps by bounding the dilatation at any point of $M$ by $K < \infty$. Let $\mathcal{F}$ be the family of all $K$-quasiconformal homeomorphisms of $M$. If for any points $p, q \in M$, there is a  map $f\in\mathcal{F}$ such that $f(p) = q$, that is, the family $\mathcal{F}$ is transitive, then $M$ is said to be $K$-quasiconformally homogeneous. Quasiconformal homogeneity was first studied by Gehring and Palka in \cite{gehring} in 1976 for genus zero surfaces and analogous higher dimensional manifolds. Gehring and Palka also showed that the only 1-quasiconformally homogeneous (i.e. $\mathcal{F}$ is transitive with all maps conformal) genus zero surfaces are non-hyperbolic. It was also found that there do exist genus zero surfaces which are $K$-quasiconformal for $1 < K < \infty$. 

A more recent question is whether there exists a uniform lower bound on $K$ for $K$-quasiconformally homogeneous hyperbolic manifolds. Using Sullivan's Rigidity Theorem, Bonfert-Taylor, Canary, Martin, and Taylor showed in \cite{bonfert} that for dimension $n \geq 3$, there exists such a universal constant $\mathcal{K}_n > 1$ such that for any $K$-quasiconformally homogeneous hyperbolic $n$-manifold other than $\mathbb{D}^n$, we have $K \geq \mathcal{K}_n$. In \cite{bonfert} it is also shown that $K$-quasiconformally homogeneous hyperbolic $n$-manifolds for $n \geq 3$ are precisely the regular covers of closed hyperbolic orbifolds. For two-dimensional surfaces, such a classification is shown to be false with the construction of $K$-quasiconformally homogeneous surfaces which are not quasiconformal deformations of regular covers of closed orbifolds in \cite{exoticqc}. 

In dimension two, Bonfert-Taylor, Bridgeman, Canary, and Taylor showed the existence of such a bound for a specific class of closed hyperbolic surfaces which satisfy a fixed-point condition in \cite{fixedpoint}. Bonfert-Taylor, Martin, Reid, and Taylor showed in \cite{strongqc} the existence of a similar bound $\mathcal{K}_c > 1$ such that if $M \neq \mathbb{H}^2$ is a $K$-strongly quasiconformally homogeneous hyperbolic surface, that is, each member of the transitive family of $K$-quasiconformally homogeneous maps is homotopic to a conformal automorphism of $M$, then $K \geq \mathcal{K}_c$. Kwakkel and Markovic proved the conjecture of Gehring and Palka for genus zero surfaces by showing the existence of a lower bound on $K$ for hyperbolic genus zero surfaces other than $\mathbb{D}^2$ in \cite{qchom}. Additionally, it was shown by Kwakkel and Markovic that for surfaces of positive genus, only maximal surfaces can be $K$-quasiconformally homogeneous (Proposition 2.6 of \cite{qchom}). 

Here, we consider a special case of the problem for closed hyperbolic Riemann surfaces of arbitrary genus.  Recall that the mapping class group of a surface $M$, denoted MCG($M$), consists of homotopy classes of orientation-preserving homeomorphisms of $M$. In general, $K$-quasiconformal homeomorphisms are representatives of any mapping class in MCG($M$). The Torelli subgroup $\mathcal{I}(M) \leq$ MCG($M$) contains those elements of MCG($M$) which act trivially on $H_1(M;\mathbb{Z})$, the first homology group of $M$ (see e.g. \S2.1 and \S7.3 of \cite{primer}). This means that the image of any closed curve $c \subset M$ under a Torelli map must be some curve homologous to $c$. We define a closed Riemann surface $M$ to be {\it Torelli-$K$-quasiconformally homogeneous} if it is $K$-quasiconformally homogeneous and there exists transitive family $\mathcal{F}$ of $K$-quasiconformal homeomorphisms which consists of maps whose homotopy classes are in the Torelli group of $M$. That is, all $f \in \mathcal{F}$ are homologically trivial. Farb, Leninger, and Margalit found bounds on the dilatation of related pseudo-Anosov maps on a Riemann surface in \cite{dilatation}. In particular, they give the following result as Proposition 2.6: 

\begin{proposition}
\label{prop2.6}
If $g \geq 2$, then $L(\mathcal{I}(M)) > .197$, where $L$ is the logarithm of the minimal dilatation of pseudo-Anosov maps in $\mathcal{I}(M)$. 
\end{proposition}

This result relies on several lemmas for pseudo-Anosov maps that can also be applied to Torelli-$K$-quasiconformal maps once we show that appropriate $K$-quasiconformal maps must exist. After proving a proposition that allows us to avoid the assumption of pseudo-Anosov maps, we will use an argument similar to that in the proof of Proposition \ref{prop2.6} to give the following result on Torelli-$K$-quasiconformally homogeneous surfaces: 

\begin{theorem} \label{main}
There exists a universal constant $K_T > 1$ such that if $M$ is a Torelli-$K$-quasiconformally homogeneous closed hyperbolic Riemann surface, then $K \geq K_T$. 
\end{theorem}

A related proof of the result has very recently been given by Vlamis, along with similar results for other subgroups of MCG($M$) in \cite{vlamis}. It would be interesting to find an actual value for $K_T$, and perhaps exhibit Torell-$K$-quasiconformally homogeneous surfaces with minimal $K$. In addition to bettering our estimate for $K_T$, proving the existence of a bound on $K$ for the general case of non-Torelli maps is of course still an important open question. Should it be found that such a bound must exist, another interesting result may be a comparison between the value of this bound and our bound $K_T$. We may also wish to seek more specific details on which types of surfaces can be Torelli-$K$-quasiconformally homogeneous with small values of $K$, and if we can obtain stricter bounds for different types of surfaces.  

{\bf Acknowledgements:} The author wishes to thank Professor Vladimir Markovic for providing an interesting project and for his effective mentoring. This work was done under a Ryser Summer Undergraduate Research Fellowship at Caltech. 

\section{Preliminary Notions} \label{s1}

First, we will introduce some relevant concepts, definitions, and lemmas which will be used throughout the paper. 

\subsection{Definitions}

Let $M$ be a closed hyperbolic Riemann surface of genus $g \geq 0$. Let $c$ be a shortest geodesic on $M$. The injectivity radius $\iota(M)$ is the infimum over all $p \in M$ of the largest radius for which the exponential map at $p$ is injective (see \S2.1 of \cite{qchom}). In particular, $|c| \geq 2\iota(M)$, where $|c|$ denotes the length of $c$. 

Let $\gamma\subset M$ be a closed curve. We denote by $[\gamma]$ its homotopy class. Recall that in any homotopy class, there exists a unique geodesic whose length bounds the length of all elements of $[\gamma]$ from below (see e.g. \cite{thurston}). We define the geometric intersection number of two closed curves $a, b \subset M$ by 
\begin{equation}
i(a, b) = \min \# \{ \gamma \cap \gamma' \},
\end{equation}
where the minimum is taken over all closed curves $\gamma, \gamma' \subset M$ with $[\gamma] = a$ and $[\gamma'] = b$ (as defined in \cite{qchom}). From the discussion on page 804 of \cite{dilatation}, the intersection number of a pair of homologous curves must be even. 

Next, let $X$ and $Y$ be complete metric spaces. A map $f: X \rightarrow Y$ is said to be a \textit{$K$-quasi-isometry} if for some $R>0$ we have
$$
R+Kd(x,x') \geq d(f(x), f(x')) \geq \frac{d(x,x')}{K} - R
$$
for all $x, x' \in X$ (see \cite{ratcliffe}). A \textit{quasi-geodesic} in a metric space $X$ is a quasi-isometric map
$$
\gamma: [a,b] \rightarrow X.
$$
It is known that the image of a geodesic under a quasiconformal homeomorphism is a quasi-geodesic (Theorem 5.1 of \cite{epstein}), and that a quasi-geodesic is within a bounded distance of a unique hyperbolic geodesic. 

\subsection{Previous Results}

Let $M$ be as above, and recall that the Torelli group of $M$ is denoted $\mathcal{I}(M)$. First, we have Lemmas 2.2, 2.3, and 2.4 respectively from \cite{dilatation}. These will be instrumental in our final proof in Section \ref{mainproof}. 

\begin{lemma} \label{lemma2.2}
Suppose that $[f] \in \mathcal{I}(M)$, that $c$ is a separating curve, and that $[f(c)] \neq [c]$. Then $i(f(c), c) \geq 4$. 
\end{lemma}

\begin{lemma} \label{lemma2.3}
Suppose that $[f] \in \mathcal{I}(M)$, that $c$ is a nonseparating curve, and that $[f(c)] \neq [c]$. Then at least one of $i(f(c), c)$ and $i(f^2(c), c)$ is at least 2. 
\end{lemma}

\begin{lemma} \label{lemma2.4}
 Suppose $c$ and $c'$ are homologous nonseparating curves with $i(c, c') = 2$. Suppose that $d, d', e,$ and $e'$ are the boundary components of a 4-holed sphere as shown in Figure 1. Then $d$ and $d'$ are separating in $M$, and $[e] = -[e] = [c] = [c']$ in $H_1(M;\mathbb{Z})$. 
\end{lemma}

\begin{figure}[htbp]
	\centering
	\includegraphics[width = .6 \textwidth] {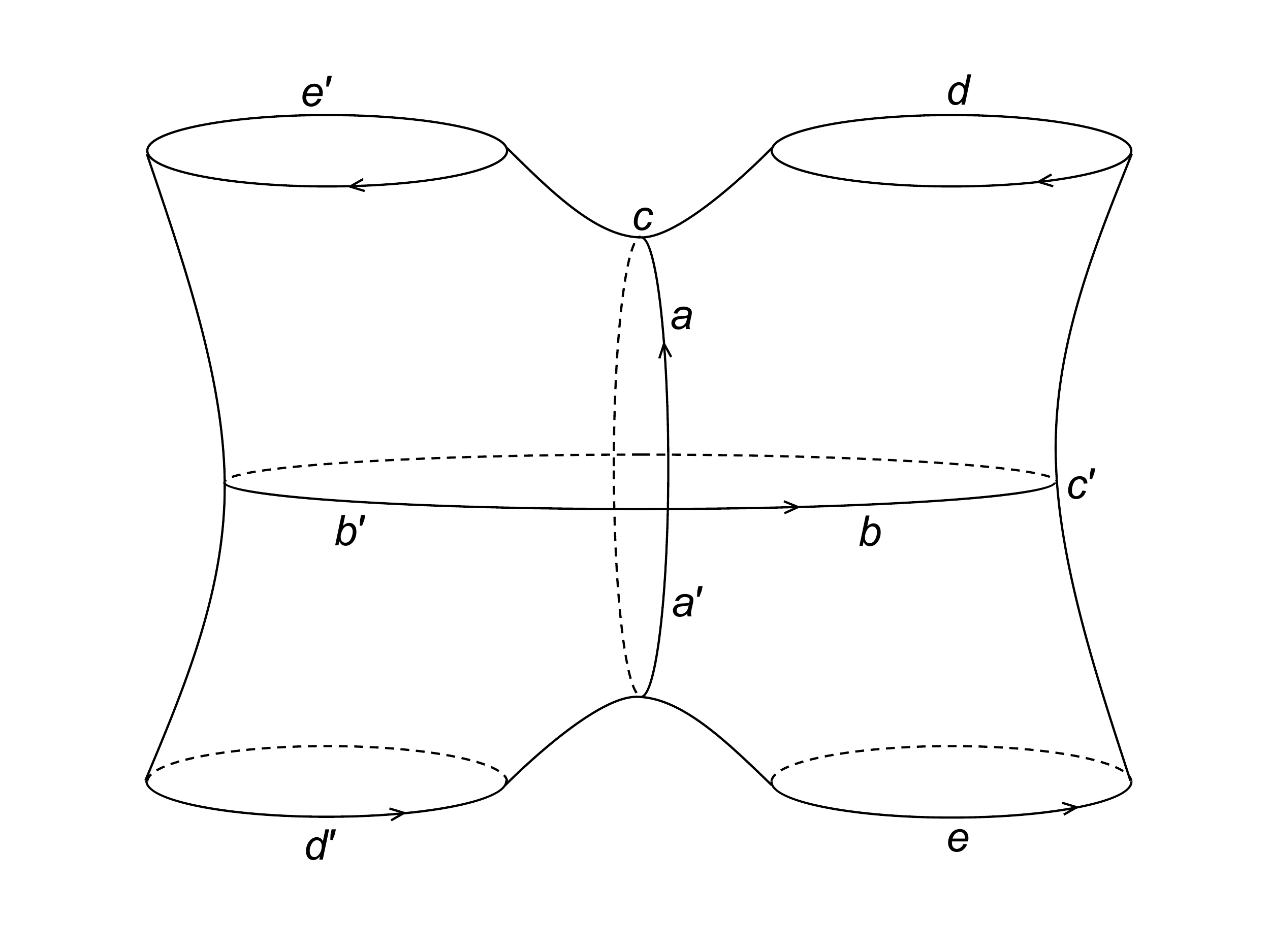}
	\caption{Illustration of Lemma \ref{lemma2.4}, with $c$ and $c'$ homologous with intersection number 2. The labels $a, b, a', $ and $ b'$ are used in the proof of Theorem \ref{main}. (Adapted from \cite{dilatation})}
	\label{4holesph}
\end{figure}

Next, we have Lemmas 2.2 and 2.3 from \cite{qchom}, respectively:

\begin{lemma} \label{lemma2.2qchom}
Let M be a K-quasiconformally homogeneous hyperbolic surface and $\iota(M)$ its injectivity radius. Then $\iota(M)$ is uniformly bounded from below (for K bounded from above) and $\iota(M) \rightarrow \infty$ for $K \rightarrow 1$. 
\end{lemma}

In particular, if $c$ is the shortest curve on $M$, then $|c| \rightarrow \infty$ as $K \rightarrow 1$. 

\begin{lemma} \label{lemma2.3qchom}
Let $\gamma$ a simple closed geodesic in $M$. Let $f: M\rightarrow M$ be a K-quasiconformal homeomorphism and $\gamma '$ the simple closed geodesic homotopic to $f(\gamma)$. Then $\frac{1}{K}|\gamma| \leq |\gamma '| \leq K|\gamma|$. 
\end{lemma}

This will allow us to bound the lengths of geodesics under $K$-quasiconformal maps in terms of their pre-images. We also recall the following classical result: 

\begin{proposition} \label{geoNBH}
There exists a function $\delta(K) > 0$ such that $\delta(K) \rightarrow 0$ as $K \rightarrow 1$, which satisfies the following. Let $f: S_1 \rightarrow S_2$ be a $K$-quasiconformal map between two hyperbolic Riemann surfaces and suppose that $\gamma$ is a geodesic on $S_1$. Then $f(\gamma)$ is contained in a $\delta(K)$-neighborhood of the unique geodesic on $S_2$ homotopic to $f(\gamma)$. 
\end{proposition}

 That is, the image of a geodesic under a $K$-quasiconformal map is contained within a collar of a geodesic, the width of which tends to 0 as $K$ tends to 1. Finally, Proposition 1.16 from \cite{buser} gives us the following: 

\begin{proposition} \label{shortestcurve}
Let $S$ be a closed Riemann surface of genus $g \geq 2$ equipped with its hyperbolic metric. Then the shortest curve $c$ on $S$ has length $|c| \leq 2\log(4g-2)$.
\end{proposition}

These two facts will allow us to prove the main proposition in the next section.

\section{Existence of a Suitable $f \in \mathcal{F}$}

Let $S$ be a closed hyperbolic Riemann surface with shortest geodesic $c$. Suppose that $S$ is $K$-quasiconformally homogeneous with transitive family of $K$-quasiconformal maps $\mathcal{F}$. 

In \cite{dilatation}, a lower bound was obtained for the dilatation of pseudo-Anosov maps. Recall the condition in Lemmas \ref{lemma2.2} and \ref{lemma2.3}, that we have some map $f$ such that $[c] \neq [f(c)]$. Pseudo-Anosov maps are always homotopically nontrivial, so the existence of such an $f$ is known a priori. For general $K$-quasiconformal homeomorphisms, this is not necessarily the case, but our proof of Theorem \ref{main} makes use of the aforementioned lemmas with the curve $c$, as well as a nearby geodesic. Thus, we must show that there exists a map $f \in \mathcal{F}$ that sends both $c$ and a neighboring geodesic to curves not homotopic to their preimages for surfaces with sufficiently small $K$. We phrase this as follows: 

\begin{proposition} \label{existence}
There exists a universal constant $K_0 > 1$ such that if $S$ is a $K_0$-quasiconformally homogeneous closed Riemann surface of genus $g \geq 2$, and $\mathcal{F}$ is a transitive family of $K_0$-quasiconformal homeomorphisms of $S$, then if $c$ is the shortest geodesic on $S$ and $d$ another geodesic on $S$ whose length is at most $2|c|$ and such that the distance between $c$ and $d$ on $S$ is at most $\frac{1}{|c|}$, there exists $f \in \mathcal{F}$ such that $[c] \neq [f(c)]$ and $[d] \neq [f(d)]$. 
\end{proposition}

We will prove this in three parts. First, we show that surfaces lacking a $f \in \mathcal{F}$ such that $[f(c)] \neq [c]$ and $[f(d)] \neq [d]$ must be contained in the union of small neighborhoods of $c$ and $d$. Then, we will exhibit a bound on the area of such a surface. Finally, we show that for sufficiently small $K$ this bound cannot hold. 

\begin{claim} \label{neighborhood}
Let $S$ be as in Proposition \ref{existence}. If for all $f \in \mathcal{F}$, we have that $f$ fixes at least one of $[c]$ or $[d]$, then S is contained in the union of a $\delta$-neighborhood of $c$ and a $\delta$-neighborhood of $d$, where $\delta = \delta(K)$ depends on $K$ and $\delta \rightarrow 0$ as $K \rightarrow 1$. 
\end{claim}

\begin{proof}
By hypothesis, we can choose some $x\in S$ such that $x$ is in a $\frac{1}{|c|}$-neighborhood of both $c$ and $d$. Since $\mathcal{F}$ is transitive, for any point $y \in S$ we can find a map $f \in \mathcal{F}$ such that $f(x) = y$. By Proposition \ref{geoNBH}, we know that each $f \in \mathcal{F}$ sends any point on $c$ or $d$ to a point contained in a $\delta^*(K)$-neighborhood of a geodesic, where $\delta^* \rightarrow 0$ as $K \rightarrow 1$. By continuity, we know that each image of $x$ will be in a slightly larger neighborhood of a geodesic, say a $\delta(K)$-neighborhood. Notice that as $|c|$ increases (as $K \rightarrow 1$ by Lemma \ref{lemma2.2qchom}), we have $\delta \rightarrow \delta^*$ because $c$ and $d$ are separated by a distance $\frac{1}{|c|}$. 

Now, if all maps $f \in \mathcal{F}$ fix at least one of $[c]$ or $[d]$, the image of $x$ must remain in a $\delta$-neighborhood of at least one of these curves. It follows that $S$ is contained in the union of a $\delta$-neighborhood of $c$ and a $\delta$-neighborhood of $d$, where $\delta$ depends only on $K$. By Proposition \ref{geoNBH}, together with Lemma \ref{lemma2.2qchom}, we can make $\delta$ arbitrarily small by sending $K\rightarrow 1$. These neighborhoods are collars around the curves $c$ and $d$ of total width $2\delta$, and as remarked above, sending $K\rightarrow 1$ will send $\delta \rightarrow 0$. This completes the proof of the claim. 
\end{proof} 
 
In the rest of the proof, we show that when $\delta$ is small, $S$ cannot be contained in these two $\delta(K)$-neighborhoods of $c$ and $d$. 

\begin{claim} \label{areaboundclaim}
Let $S$ be as above. Then:
$$
\mathrm{Area}(S) < 2\pi(\frac{3\log(4g-2)}{\delta} + 2)(\cosh(2\delta) - 1).
$$
\end{claim}

\begin{proof}
Since $S$ is contained in small neighborhoods of the two curves $c$ and $d$, we will bound the area of the neighborhoods of each curve from above as follows. Each collar can be covered by hyperbolic disks (2-balls) of radius $2\delta$ whose centers lie on the main curve (See Figure 2). We arrange them such that each disk is separated by a distance of $2\delta$ from the adjacent disks. That gives a total of less than $\frac{|c|}{2\delta} + 1$ disks (taking the smallest integer greater than $\frac{|c|}{2\delta}$) for curve $c$, and $\frac{|d|}{2\delta} + 1$ disks for $d$. One disk for each of $c$ and $d$ will be less than $2\delta$ away from one of its neighbors if $|c|$ or $|d|$ is not an integer multiple of $2\delta$.

\begin{figure}[htbp]
	\centering
	\includegraphics[width = .8 \textwidth] {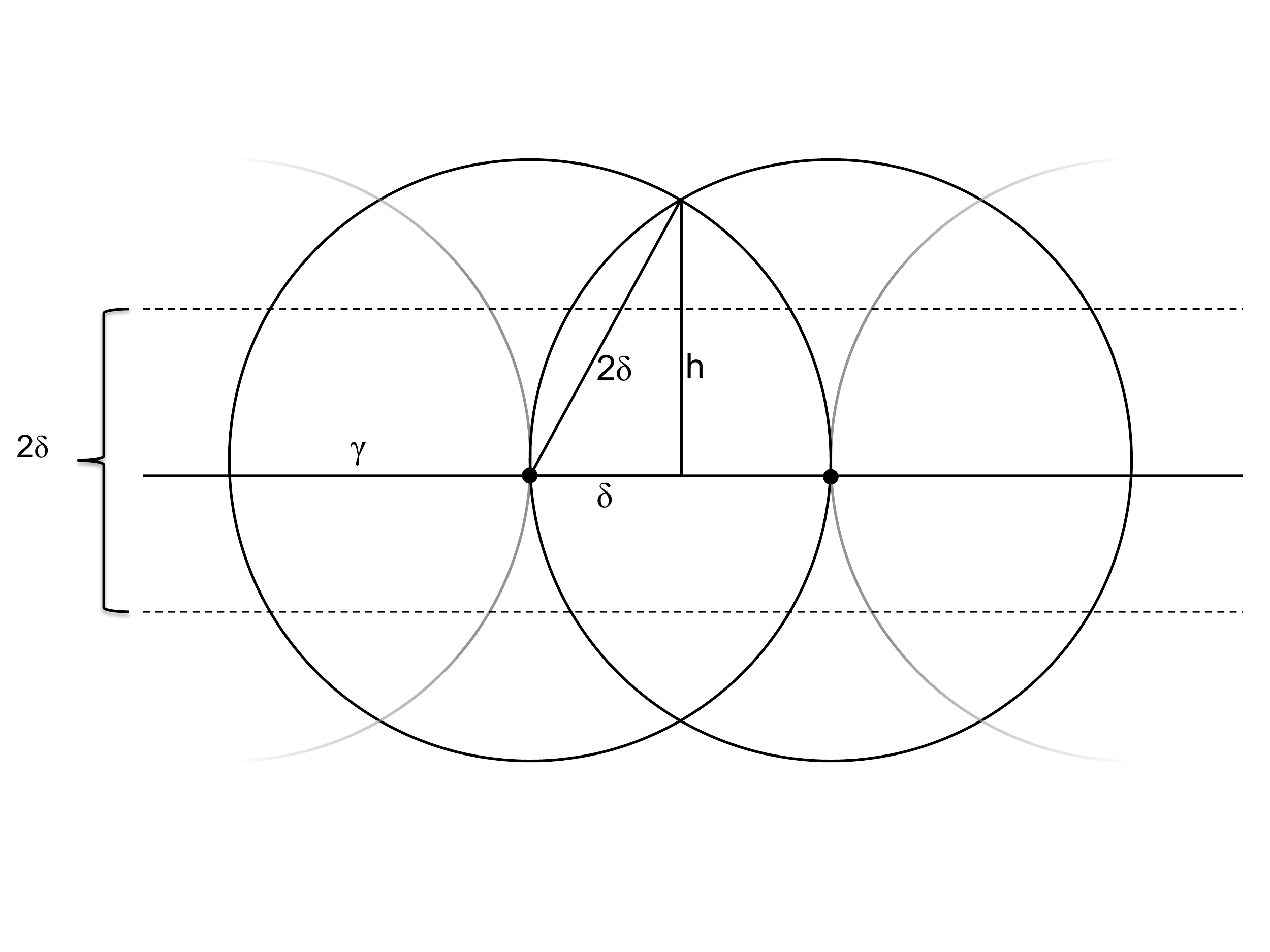}
	\label{nbh}
	\caption{Covering the $\delta$-neighborhood of the curve $\gamma$ with disks of radius $2\delta$. The hyperbolic right triangle has base $\delta$ (half the separation between the disks) and height $h\geq\delta$. The hypotenuse is the radius of a disk. We do this for curves $c$ and $d$ }
\end{figure}

In order to show that the disks cover the entire collar, we must show that the height $h$ of the hyperbolic right triangle (i.e. half of the width of the area covered by the disks) in Figure 2 is at least $\delta$. In that case, the disks will cover a collar around their respective curve of total width at least $2\delta$ (since the disks are convex), thus covering the $\delta$-neighborhood of the curve. This follows from the hyperbolic Pythagorean theorem, which gives 
$$
\cosh(2\delta) = \cosh(\delta)\cosh(h).
$$
Indeed, supposing otherwise and applying the appropriate identities, we have
$$
h < \delta \Rightarrow \cosh(h) < \cosh(\delta) \Rightarrow \frac{\cosh(2\delta)}{\cosh(\delta)} < \cosh(\delta)
$$
$$
\Rightarrow 2\cosh^2(\delta) - 1 < \cosh^2(\delta) \Rightarrow \cosh^2(\delta) < 1,
$$

contradicting the fact that $\cosh(x) \geq 1$ for all $x \in \mathbb{R}$. Thus, the disks cover the $\delta$-neighborhoods of $c$ and $d$, whence they cover $S$. 

Recall that the area of a hyperbolic disk of radius $r$ is $2\pi (\cosh{r} - 1)$. Since our collection of disks bounds the area of $S$ from above, we have: 

\begin{equation} \label{prelimareabound}
\mathrm{Area}(S) < (\frac{|c|}{2\delta} + 1)2\pi(\cosh(2\delta) -1) + (\frac{|d|}{2\delta} + 1)2\pi(\cosh(2\delta) -1)
\end{equation}

where $(\frac{|c|}{2\delta} + 1)$ and $(\frac{|d|}{2\delta} + 1)$ bound the number of disks from above, and $2\pi(\cosh(2\delta)-1)$ is the area of each disk. Since by hypothesis we have that $|d| \leq 2|c|$, we can rewrite this as: 

\begin{equation} \label{areabound}
\mathrm{Area}(S) <  (\frac{3|c|}{2\delta} + 2)2\pi(\cosh{2\delta} -1)
\end{equation}

From Proposition \ref{shortestcurve} we also have that $|c| \leq 2\log{(4g-2)}$. Together with (\ref{areabound}), this gives: 

\begin{equation} \label{bound_ineq}
\mathrm{Area}(S) < 2\pi(\frac{3\log(4g-2)}{\delta} + 2) (\cosh{2\delta} -1),
\end{equation}

as desired. 
\end{proof} 

Armed with the inequality (\ref{bound_ineq}), we can proceed to the final proof of Proposition \ref{existence}. 

\begin{proof}[Proposition \ref{existence}]
Let $S$ be as above. We need to show that there exists a universal constant $K_0 > 1$ such that if $K \leq K_0$, then for some $f \in \mathcal{F}$, we have both $[f(c)] \neq [c]$ and $[f(d)] \neq [d]$. We show that $K$ is bounded from below by some $K_0 > 1$ for surfaces with families in which no such $f$ exists. Recall that the area of a hyperbolic surface $S$ of genus $g\geq2$ is given by:

\begin{equation} \label{area}
\mathrm{Area}(S) = 4\pi(g-1). 
\end{equation}

Now, from the previous claim we have an upper bound for the area of our surface $S$ in terms of the genus $g$ and $\delta = \delta(K)$. Combining (\ref{bound_ineq}) and (\ref{area}), we have: 

\begin{equation} \label{main_ineq}
4 \pi (g - 1) < 2\pi(\frac{3\log(4g-2)}{\delta} + 2) (\cosh(2\delta) -1).
\end{equation}

This inequality follows from the upper bound on the area from Claim \ref{areaboundclaim}, the area of a hyperbolic surface from (\ref{area}), and the upper bound on the lengths of $c$ from Proposition (\ref{shortestcurve}). Simplifying, we obtain: 

\begin{equation} \label{mid_ineq}
\frac{2(g-1)}{\log(4g-2)} < (\frac{3}{\delta} + \frac{2}{\log(4g-2)})(\cosh(2\delta) - 1) < (\frac{3}{\delta} +2)(\cosh(2\delta) - 1).
\end{equation}

We have $g \geq 2$, and so (\ref{mid_ineq}) gives:

\begin{equation} \label{final_ineq}
\frac{2}{\log(6)}  < (\frac{3}{\delta} +2)(\cosh(2\delta) - 1).
\end{equation}

Notice that for the upper bound in (\ref{final_ineq}), we have

\begin{equation} \label{deltalimit}
\lim_{\delta \rightarrow 0} (\frac{3}{\delta} +2)(\cosh(2\delta) - 1) = 0.
\end{equation}

Now, (\ref{final_ineq}) and (\ref{deltalimit}) show that there exists a uniform lower bound on $\delta$. By Lemma \ref{lemma2.2qchom} and Proposition \ref{geoNBH}, we can choose $K>1$ such that $\delta$ becomes arbitrarily small, which sends the right-hand side of (\ref{final_ineq}) to 0. Thus, there must be a universal lower bound $K_0 > 1$ on $K$. If the transitive family $\mathcal{F}$ does not include maps that send $c$ to non-homotopic curves, then $K > K_0$. 

Thus for $1< K \leq K_0$, the transitive family $\mathcal{F}$ must include a map that sends both $c$ and $d$ to a non-homotopic curve. 

\end{proof} 

Using Proposition \ref{existence}, will now prove Theorem \ref{main}. 

\section{Proof of Theorem \ref{main}} \label{mainproof}

\begin{proof}
Let $S$ be a closed hyperbolic Riemann surface, and suppose $S$ is Torelli-K-quasiconformally homogeneous with family of homeomorphisms $\mathcal{F}$. Suppose $1 < K \leq K_0$ from Proposition \ref{existence}. Let $c$ be a simple closed geodesic on $S$ of minimal length. Using Proposition \ref{existence}, choose some $f \in \mathcal{F}$ such that $[c] \neq [f(c)]$ and $f$ is also homotopically non-trivial on any geodesic in a $\frac{1}{|c|}$-neighborhood of $c$. 

We have the following two cases: (1) $i(c, f(c)) \geq 4$ or $i(c, f^2(c)) \geq 4$, or (2) $c$ is nonseparating and $i(c, f(c))$ and $i(c, f^2(c))$ are both less than 4. Notice that these cover all possibilities: if $c$ is separating, then Lemma \ref{lemma2.2} gives us that $i(c, f(c)) \geq 4$. If $c$ is nonseparating, then either one of $i(c, f(c))$ or $i(c, f^2(c))$ is greater than 4, which is case 1, and otherwise we have case 2. 
\newline

{\it Case 1}: Let $h$ be either $f$ or $f^2$, where $i(c, h(c)) \geq 4$. Since $i(c, h(c)) \geq 4$, any member of the homotopy class $[h(c)]$ will have at least 4 intersections with $c$. In particular, let $c'$ be the geodesic homotopic to $h(c)$. The intersection points $c \cap c'$ cut $c$ and $c'$ into arcs. Since there are at least 4 such points, there is an arc $a$ of $c'$ which satisfies
\begin{equation}
|a| \leq \frac{|c'|}{4} \leq \frac{K^2 |c|}{4}.
\end{equation}
The second inequality follows from Lemma \ref{lemma2.3qchom}, with $K^2$ since $h$ is possibly $f^2$, and $f^2$ is a $K^2$-quasiconformal homeomorphism. The endpoints of $a$ cut $c$ into two arcs, one of which, say $b$, has length $|b| \leq |c|/2$. The union $a \cup b$ is a simple closed curve. It must be homotopically nontrivial since otherwise we could, by homotopy, reduce the number of intersections of $c$ and $c'$ below $i(c, c')$. Now, recall $c$ is the shortest closed geodesic, so $|c| \leq |a\cup b|$. Then we have: 

\begin{equation}
|c| \leq |a| + |b| \leq \frac{K^2 |c|}{4} + \frac{|c|}{2}
\Rightarrow 1 \leq \frac{K^2}{4} + \frac{1}{2}  
\end{equation}

Thus, $2 \leq K^2 \Rightarrow K \geq \sqrt{2}$. 
\newline

{\it Case 2}: By Lemma \ref{lemma2.3}, either $i(c, f(c)) = 2$ or $i(c, f^2(c)) = 2$. Let $h$ be either $f$ or $f^2$, where $i(c, h(c)) = 2$. Now, let $c'$ be the geodesic homotopic to $h(c)$; we still have $i(c, c') = 2$. Let $d$ and $d'$ be the separating curves from Lemma \ref{lemma2.4} with $c$ and $c'$ as the homologous pair (since $f$ is Torelli). Alternatively, the intersection points $c \cap c'$ define two arcs of $c$, say $a$ and $a'$, and two arcs of $c'$, say $b$ and $b'$ as in Figure 1. The curves $d$ and $d'$ are then $d \sim a\cup b$ and $d' \sim a' \cup b'$. Now,

\begin{equation} 
|d| + |d'| = |a| + |b| + |a'| + |b'| = |c| + |c'| \leq |c| + K^2|c|
\end{equation}

by Lemma \ref{lemma2.3qchom} and since $h$ may be $f^2$, which is a $K^2$-quasiconformal homeomorphism. It follows that at least one of $d$ and $d'$, say $d$, has length bounded above by half of $|c| +|c'|$:

\begin{equation} \label{dbound}
|d| \leq \frac{|c| + K^2|c|}{2}
\end{equation}

We now consider $d_1$, the geodesic homotopic to $d$, which is a separating curve, and the geodesic $e_1$ homotopic to $e$. We have $|d_1| \leq |d|$ and $|e_1| \leq |e|$. To continue, we require the following lemma, which will be proven at the end: 

\begin{lemma} \label{cd_dist}
For sufficiently small $K$, the curves $c$ and $d_1$ are within a distance of $\frac{1}{|c|}$ from each other.
\end{lemma}

Now, suppose $K>1$ is small enough so Lemma \ref{cd_dist} can be used. Applying the lemma and Proposition \ref{existence}, we see that that $[f(d_1)] \neq [d_1]$. We can now apply Lemma \ref{lemma2.2}, so that $i(d_1, f(d_1))$ is at least 4. As in Case 1, if $\tilde a$ is the shortest arc of the geodesic homotopic to $f(d_1)$ cut off by $d_1$ and $\tilde b$ is the shortest arc of $d_1$ cut off by $\tilde a$, then

\begin{equation} \label{case1ond}
|\tilde a\cup \tilde b| \leq |d_1|(\frac{K}{4} + \frac{1}{2}).
\end{equation}

Note that we can use $K$ instead of possibly $K^2$ since $d$ is separating, so Lemma \ref{lemma2.2} applies with $f$. Recall that we also have 

\begin{equation} \label{lambdaboundcase2}
|c| \leq |\tilde a\cup \tilde b|. 
\end{equation}

Combining (\ref{dbound}), (\ref{case1ond}), (\ref{lambdaboundcase2}), and the fact that $|d_1| \leq |d|$, we see that

\begin{equation}
|c| \leq (\frac{K}{4} + \frac{1}{2})(\frac{|c| + K^2|c|}{2}) 
\end{equation}

Which then gives 

\begin{equation}
1 \leq (\frac{K}{4} + \frac{1}{2})(\frac{1 + K^2}{2}) \Rightarrow K^3 + 2K^2 + K - 6 \geq 0.
\end{equation}

The cubic polynomial in $K$ has one real root, and so 

\begin{equation}
K \geq -\frac{2}{3} + \frac{1}{3}\sqrt[3]{82 - 9\sqrt{83}} + \frac{1}{3}\sqrt[3]{82 + 9\sqrt{83}} \approx 1.218
\end{equation}

approximated from below. 
\newline

Now, since these bounds are independent of genus, we have that $K > 1.218$. We can then see that for some $K_T > 1$, any such surface $S$ cannot be Torelli-$K$-quasiconformally homogeneous with $K \leq K_T$. Thus we have a universal constant $K_T>1$ such that any Torelli-$K$-quasiconformally homogeneous closed hyperbolic Riemann surface of genus $g \geq 2$ must have $K \geq K_T > 1$. This completes the proof. 

\end{proof} 

\begin{proof} [Proof of Lemma \ref{cd_dist}]
First, consider the pair of pants defined by curves $c$, $d_1$, and $e_1$, as in Figure \ref{4holesph}. Notice that as $K$ approaches 1, we eventually have: 
\begin{equation} \label{d1e1bound}
|c| \leq |d_1| \leq \frac{3|c|}{2}, |c| \leq |e_1| \leq \frac{3|c|}{2}. 
\end{equation}
This follows from (\ref{dbound}), which can also apply with the same bounds to the curve $e_1$, and the fact that $c$ is the shortest geodesic on $S$. 

Consider now the right-angled hyperbolic hexagon formed by inserting perpendiculars connecting each of $c$, $d_1$, and $e_1$, and cutting off a fundamental hexagon from this pair of pants. Then there are three alternating sides of length $\frac{|c|}{2}$, $\frac{|d_1|}{2}$, and $\frac{|e_1|}{2}$. By Lemma \ref{lemma2.2qchom} we also have that as $K\rightarrow1$, $|c|\rightarrow\infty$. Recall from \cite{ratcliffe} the Law of Cosines for right-angled hyperbolic hexagons: 
\begin{equation} \label{hexagon}
\cosh(z') = \coth(x)\coth(y) + \frac{\cosh(z)}{\sinh(x)\sinh(y)},
\end{equation}
where $x$, $y$, and $z$ are the lengths of alternate sides of the hexagon, and $z'$ is the length of the side opposite the side of length $z$. Let $x$, $y$, and $z$ correspond to the sides of length $\frac{|c|}{2}$, $\frac{|d_1|}{2}$, and $\frac{|e_1|}{2}$ respectively, and so $z'$ is the distance between curves $c$ and $d_1$. We wish to show that $z' \rightarrow 0$ faster than $\frac{1}{|c|} \rightarrow 0$ as $K\rightarrow 1$. Notice from (\ref{d1e1bound}) that, as $|c| \rightarrow \infty$, $x$, $y$, and $z$ all increase within a factor of $\frac{3}{2}$ from each other. Thus, the $\coth(x)\coth(y)$ term proceeds exponentially to 1, and the $\frac{\cosh(z)}{\sinh(x)\sinh(y)}$ term exponentially approaches 0. 

Now, the right-hand side of (\ref{hexagon}) proceeds to unity, and so $\cosh(z') \rightarrow 1$. We can then see that $z'$, the distance between curves $c$ and $d_1$, approaches zero exponentially in $|c|$ as $|c|$ increases. Therefore as $K \rightarrow 1$, the curves $c$ and $d_1$ will approach each other exponentially in $|c|$, and so for sufficiently small $K$, they will be closer than $\frac{1}{|c|}$, as desired. 

\end{proof}

\end{document}